\documentclass[11pt]{amsart}
\pdfoutput=1
\usepackage{amsmath, amssymb}
\usepackage{amsfonts}
\usepackage{mathrsfs}
\usepackage[arrow,matrix,curve,cmtip,ps]{xy}
\usepackage{graphicx}
\usepackage{float}
\usepackage{amsthm}

\allowdisplaybreaks

\newtheorem{theorem}{Theorem}[section]
\newtheorem{lemma}[theorem]{Lemma}

\newtheorem{corollary}[theorem]{Corollary}

\newtheorem*{theorem*}{Theorem}
\theoremstyle{remark}

\newtheorem{question}{Question}

\numberwithin{equation}{section}

\newcommand{\C}{\mathbb{C}}

\renewcommand{\C}{\mathcal{C}}

\begin{document}
\title[Optimal intersection]{Optimal intersection growth along invariant quasi-axes in the curve complex}

\title{Small intersection numbers in the curve graph}

\author{Tarik Aougab and Samuel J. Taylor}

\address{Department of Mathematics \\ Yale University \\ 10 Hillhouse Avenue, New Haven, CT 06510 \\ USA}
\email{tarik.aougab@yale.edu}

\address{Department of Mathematics \\ University of Texas at Austin \\ 1 University Station C1200, Austin, TX 78712 \\ USA }
\email{staylor@math.utexas.edu}

\date{\today}



\begin{abstract}

Let $S_{g,p}$ denote the genus $g$ orientable surface with $p \ge 0$ punctures, and let $\omega(g,p)= 3g+p-4$. We prove the existence of infinitely long geodesic rays $\left\{v_{0},v_{1}, v_{2}, ...\right\}$ in the curve graph satisfying the following optimal intersection property:  for any natural number $k$, the endpoints $v_{i},v_{i+k}$ of any length $k$ subsegment intersect $O(\omega^{k-2})$ times. By combining this with work of the first author, we answer a question of Dan Margalit.
\end{abstract}

\maketitle

\section{Introduction}
Let $S_{g,p}$ denote the orientable surface of genus $g \geq 0$ with $p\geq 0$ punctures. The \textit{curve graph} for $S_{g,p}$, denoted $\mathcal{C}_{1}(S_{g,p})$, is the graph whose vertices correspond to homotopy classes of essential, non-peripheral simple closed curves on $S_{g,p}$, and whose edges join vertices that represent curves whose union is a $2$-component multi-curve. Denote distance in this graph by $d_{\mathcal{C}(S_{g,p})}$ (or simply $d$ when the surface is clear from context). The subscript $1$ denotes the fact that $\mathcal{C}_{1}$ is naturally the $1$-skeleton of a $3g+p-4$-dimensional flag simplicial complex, in which the $k$-simplices correspond to $(k+1)$-component multi-curves. We denote by $\mathcal{C}_{0}$ the vertices of the graph $\mathcal{C}_{1}$.

By an argument going back to Lickorish \cite{Lickorish} and stated explicitly by Hempel \cite{Hempel}, the geometric intersection number strongly controls the distance $d_{\mathcal{C}}$; concretely, given a pair of curves $\alpha,\beta$ on $S_{g,p}$, 
\[ d_{\mathcal{C}}(\alpha,\beta)\leq 2\log_{2}(i(\alpha,\beta))+2.\]

A complexity-dependent version of this bound was obtained by the first author \cite{Aougab1}; in what follows, let $\omega(S_{g,p})=\omega(g,p)=3g+p-4$. Then for any $\lambda \in (0,1)$, there exists $N=N(\lambda)$ such that for all $S$ with $\omega(S)>N$, if $\alpha,\beta \in \mathcal{C}_{0}(S)$, 

\begin{eqnarray}\label{Tarik}
d_{\mathcal{C}(S)}(\alpha,\beta)\geq k \Rightarrow i(\alpha,\beta)> \omega^{\lambda(k-2)}.
\end{eqnarray}

The purpose of this note is to establish a corresponding upper bound on the minimal number of times a pair of distance $k$ simple closed curves intersect. We show:

\begin{theorem} \label{main} For any $g,p$ with $\omega(g,p)>0$, there exists an infinite geodesic ray $\gamma= \left\{v_{0},v_{1},v_{2},...\right\}$ such that for any $i\leq j$, 
\[ i(v_{i},v_{j})\le \epsilon (\epsilon B)^{2j-5}\omega^{|j-i|-2}+O\left(\omega^{|j-i|-4}\right),\]
where $B$ is a universal constant and $\epsilon =1$ if $g \ge 1$ and $\epsilon = 4$ otherwise.
\end{theorem}

For convenience, denote 
$$i_{k,(g,p)} = \min \{ i(\alpha,\beta) : \alpha,\beta \in \C(S_{g,p}), \,  d_S(\alpha,\beta) = k   \}. $$
Then Theorem \ref{main} implies $i_{k,(g,p)}$ is bounded above by a polynomial function of $\omega$ with degree $k-2$. \\

We remark that Theorem \ref{main} was proven in response to the following question, formulated by Dan Margalit:

\begin{question}[Margalit]\label{Dan}
Is it the case that for $S_g$, $i_{k,g} = O(g^{k-2})$?
\end{question}

 Combining Theorem \ref{main} with the lower bound coming from inequality \ref{Tarik} gives a positive answer to Question \ref{Dan}, asymptotically in genus.\\

\noindent \textbf{Acknowledgements.} The authors would like to thank Dan Margalit for proposing the question, and for many helpful conversations. This work was initiated during the AMS Mathematics Research Communities program on geometric group theory, June 2013.

\section{Preliminaries}

We briefly recall the definition of the curve complex for an annulus, and we review the properties of the subsurface projection to this complex. See \cite{MM2} for the general definition of subsurface projections and additional details.

For a closed annulus $Y \subset S$ whose core curve $\alpha$ is essential, let $\tilde{Y}$ be the cover of $S$ corresponding to $Y$. Denote by $\overline{Y}$ the compactification of $\tilde{Y}$ obtained in the usual way, for example by choosing a hyperbolic metric on $S$. The curve complex $\C(Y)$ is the graph whose vertices are homotopy classes of properly embedded, simple arcs of $\overline{Y}$ with endpoints on distinct boundary components.  Edges of $\C(Y)$ correspond to pairs of vertices that have representatives with disjoint interiors. The projection $\pi_Y$ from the curve complex of $S$ to the curve complex of $Y$ is defined as follows: for any $\beta \in \C(S)$ first realize $\alpha$ and $\beta$ with minimal intersection. If $\beta$ is disjoint from $\alpha$ then $\pi_Y(\beta) = \emptyset$. Otherwise, the complete preimage of $\beta$ in $\tilde{Y}$ contains arcs with well-defined endpoint on distinct components of $\partial \overline{Y}$. Define $\pi_Y (\beta) \subset \C(Y)$ to be this collection of arcs in $\overline{Y}$.

If $\alpha$ is a curve in $S$, we also denote by $\mathcal{C}(\alpha)$ the curve complex for the annulus $Y$ with core curve $\alpha$. Let $\pi_{\alpha}: \C(S) \setminus N_1(\alpha) \to \C_{\alpha}$ be the associate subsurface projection. From \cite{MM2}, we note that when $\omega(S) > 1$ the diameter of $\pi_{\alpha}(\beta)$ is $\le1$ for any curve $\beta$ that meets $\alpha$. Further, $\pi_{\alpha}$ is coarsely $1$-Lipschitz along paths in $\C(S) \setminus N_1(\alpha)$, i.e if $\gamma_0, \gamma_1, \ldots, \gamma_n$ is a path in $\C_0$ with $\pi_{\alpha}(\gamma_i) \neq \emptyset$ for each $i$, then $d_{\alpha}(\gamma_1,\gamma_n) \le n+1$.  Also recall that if $T_{\alpha}$ denotes the Dehn twist about $\alpha$ then 
$$d_{\alpha}(\gamma, T_{\alpha}^N(\gamma)) \ge N -2 .$$
Here, $d_{\alpha}(\beta,\gamma)$ is short-hand for $d_{\alpha}(\pi_{\alpha}(\beta), \pi_{\alpha}(\gamma))$.

As a consequence of the Lipschitz condition of the projection, note that if $\beta, \gamma \in \C(S)$ both meet $\alpha$ and 
$$d_{\alpha}(\beta, \gamma) \ge d_S (\beta, \gamma) +2 ,$$
then any geodesic in $\C(S)$ from $\beta$ to $\gamma$ contain a vertex disjoint from $\alpha$, i.e. any geodesic from $\beta$ to $\gamma$ must pass through $N_1(\alpha)$. In fact, a much stronger result, known as the bounded geodesic image theorem, is true.  This was first proven by Masur and Minsky in \cite{MM2}, but the version we state here is due to Webb and gives a uniform, computable constant \cite{webb2013short}. It is stated below for general subsurfaces, although we will use it only for annuli.

\begin{theorem}[Bounded geodesic image theorem] \label{BGIT}
There is a $M \ge 0$ so that for any surface $S$ and any geodesic $g$ in $\C(S)$, if each vertex of $g$ meets the subsurface $Y$ then $\mathrm{diam}(\pi_Y(g)) \le M$.
\end{theorem}

We end this section with the following well known fact, see \cite{Iv2}. Let $\mathcal{C}_{0}(S)$ denote the vertex set of $\mathcal{C}_{1}(S)$. Then if $\alpha, \beta, \gamma \in \C_0(S)$ then 
$$| i(\gamma, T_{\alpha}^N(\beta)) - N \cdot i(\alpha, \gamma) i(\alpha,\beta)| \le i(\beta,\gamma). $$
We refer to this as the twist inequality.

In the next section, we will briefly make use of the \textit{arc and curve graph} $\mathcal{AC}_{1}(S)$, a $1$-complex associated to a surface with boundary or punctures where the vertices are properly embedded arcs (modulo homotopy rel boundary) together with $\mathcal{C}_{0}(S)$, and edges correspond to pairs of vertices that can be realized disjointly on the surface; let $\mathcal{AC}_{0}(S)$ denote the vertices of $\mathcal{AC}_{1}(S)$.

A non-annular subsurface $\Sigma \subset S$ is called \textit{essential} if all of its boundary components are essential curves in $S$, and a properly embedded arc is \textit{essential} if it can not be homotoped into the boundary or a neighborhood of a puncture. Then there is a projection map $\pi^{\mathcal{AC}}:\mathcal{C}_{0}(S)\rightarrow \mathcal{P}(\mathcal{AC}_{0}(\Sigma))$, where $\mathcal{P}(\cdot)$ denotes the power set, defined as follows: a vertex $v \in \mathcal{C}_{0}$ is sent to the components of its intersection with $\Sigma$ which are essential in $\Sigma$.

\section{Minimal intersecting filling curves}\label{MIFC}
The proof of Theorem \ref{main} proceeds by beginning with curves $\alpha_3,\beta_3$ in $S_{g,n}$ that fill, i.e. $d_S(\alpha,\beta)\ge 3$, and have intersection number bounded linearly by $\omega(g,n)$. In many cases, we find $\alpha_3$ and $\beta_3$ whose intersection number is the minimal possible.

\begin{lemma} \label{minimal filling}
Given $S_{g,p}$, the following holds:
\begin{enumerate}
\item $p \le1$ and $g \neq 2, 0$, 
$$i_{3,(g,p)} = 2g -1.$$ 
\item If $p\le 2$ and $g = 2$,
$$i_{3,(g,p)}=4.$$
\item If $p \ge 2$ and $g \neq 2,0$,
$$i_{3,(g,p)} = 2g +p -2.$$
\item If $g=2$ and $p \ge 2$ even, 
$$ i_{3,(g,p)} = 2g+p-2,$$ 
and for $p \ge 3$ odd, 
$$2g+p-2\le  i_{3,(g,p)}\le 2g+p-1.$$ 
\item If $g=0$ and $p \ge 4$ even, 
$$p-2 \le i_{3,(g,p)} \le p,$$
and for $p$ odd, 
$$ i_{3,(g,p)} = p-1.$$
\end{enumerate}
\end{lemma}

\begin{proof}
The lower bounds in $(1)-(5)$ follow from an Euler characteristic argument using the observation that when $\alpha$ and $\beta$ fill, $\alpha \cup \beta$ is a $4$-valent graph whose complementary regions are either disks or punctured disks. For $(1)$, \cite{AH1} show the existence of filling pairs intersecting $2g-1$ times, which agrees with the lower bound from the Euler characteristic argument; when $g=2$, there exists a filling pair intersecting $4$ times \cite{FM}, which is best possible \cite{AH1}. 

$(3)$ and $(4)$ are obtained by the following procedure that produces a filling pair for $S_{g,p+2}$ from a filling pair for $S_{g,p}$ at the expense of two additional intersection points. Let $\alpha$ and $\beta$ be a filling pair for $S_{g,p}$; orient $\alpha$ and $\beta$, and label the arcs of $\alpha$ (resp. $\beta$) separated by intesection points from $\alpha_{1},...,\alpha_{i(\alpha,\beta)}$ (resp. $\beta_{1},...,\beta_{i(\alpha,\beta)}$) with respect to the chosen orientation, and a choice of initial arc. Suppose that the initial point of $\alpha_{k}$ coincides with the terminal point of $\beta{j}$, as seen on the left hand side of Figure $1$ below. 

\begin{figure}[H]
\centering
	\includegraphics[width=3.5in]{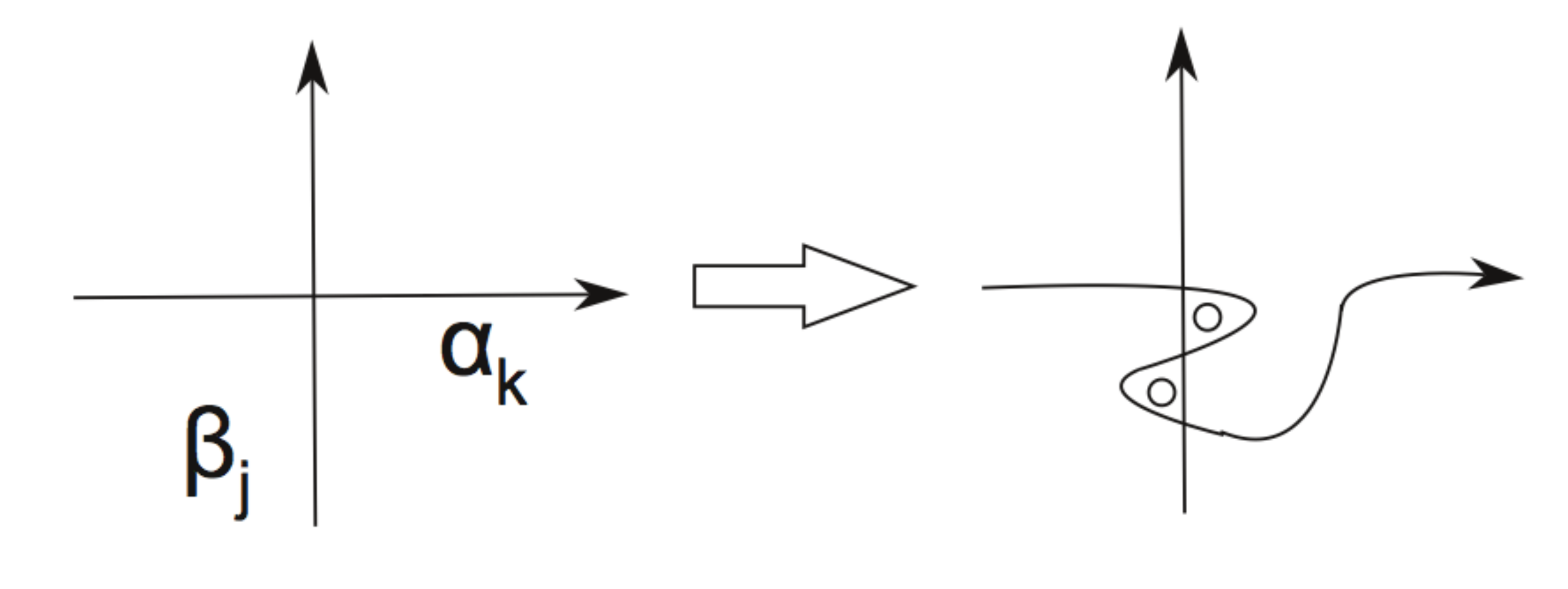}
\caption{Pushing $\alpha_{k}$ across $\beta_{j}$ and back over creates $2$ bigons; puncturing each produces a filling pair on $S_{g,p+2}$.}
\end{figure} 

Then pushing $\alpha_{k}$ across $\beta_{j}$ and back produces a pair of bigons; puncturing each of these bigons produces a filling pair intersecting $i(\alpha,\beta)+2$ times on $S_{g,p+2}$. 
Thus if $p=2k+1$ is odd and $g \neq 2$, by $(1)$ there exists a filling pair whose complement is connected, and we can puncture this single complementary region to obtain a filling pair on $S_{g,1}$. Then performing the operation pictured above k times yields a filling pair on $S_{g,p}$ intersecting $2g+p-2$ times. The Euler characteristic argument above yields a lower bound of $2g+p-2$ for $i_{3,(g,p)}$, and this proves $(3)$ in the case $p$ is odd. 

If $p$ is even, the same argument can work if there exists a filling pair $(\alpha_{g},\beta_{g})$ on $S_{g,0}$ intersecting $2g$ times, which is equivalent to the complement of $\alpha\cup \beta$ consisting of two topological disks. Assuming such a filling pair exists, we obtain a filling pair on $S_{g,2}$ intersecting $2g=2g+p-2$ times by puncturing both disks. Then the double bigon procedure described above produces the desired filling pair for any larger number of even punctures. 

Therefore, to finish the proof of $(3)$ it suffices to exhibit a filling pair on $S_{g,0}, g\neq 2$, intersecting $2g$ times. If $g=1$, take $\alpha_{1}$ to be the $(1,0)$ curve and $\beta_{1}$ the $(1,2)$ curve. 

Consider the following polygonal decomposition of $S_{2,0}$ (figure as seen in \cite{AH1}):

\begin{figure}[H]
\centering
	\includegraphics[width=3.3in]{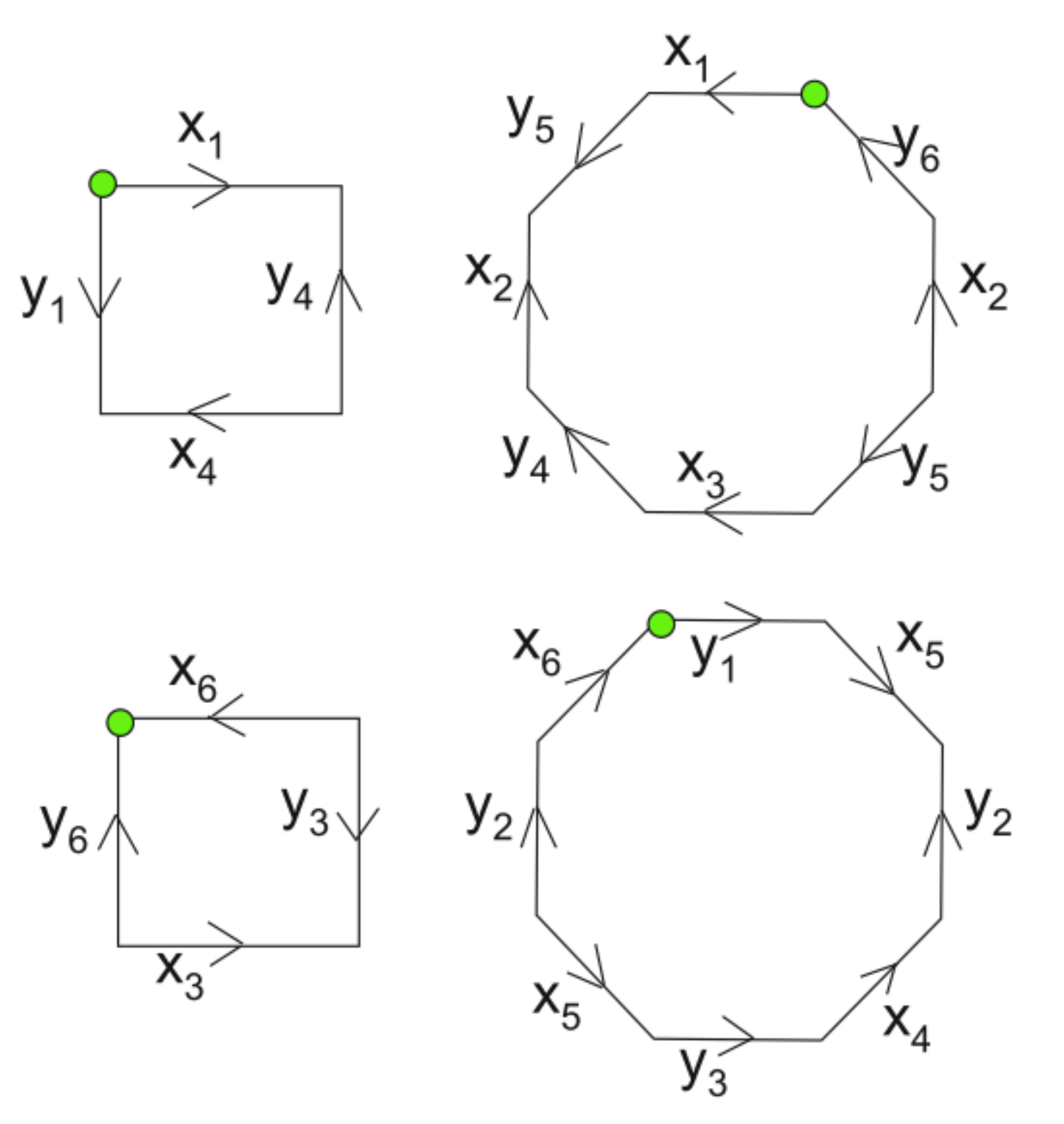}
\caption{Gluing the polygons together with respect to the oriented edge labeling yields $S_{2,0}$, and the $x$-arcs concatenate in the quotient to form a simple closed curve $x$ which fills $S_{2,0}$ with the curve $y$, the concatenation of the $y$-arcs. The $4$ green points are all identified together in $S_{2,0}$.}
\end{figure} 

The boundary of these polygons project to a filling pair $(x,y)$ on $S_{2,0}$ intersecting $6$ times. Take $S_{1,0}$, equipped with the filling pair described above intersecting twice, and cut out a small disk centered around either of these two intersection points to obtain $\tilde{S}_{1}$, a torus with one boundary component equipped with arcs $\tilde{\alpha}_{1},\tilde{\beta}_{1}$. 

Then given $S_{2,0}$ equipped with $(x,y)$, cut out a small disk centered around the green intersection point above in Figure $2$ to obtain $\tilde{S}_{2}$, a genus two surface with one boundary component equipped with arcs $\tilde{x},\tilde{y}$. Then glue $\tilde{S}_{1}$ to $\tilde{S}_{2}$ by identifying boundary components, while concatenating the endpoints of $\tilde{\alpha}_{1}$ to $\tilde{x}$, and the endpoints of $\tilde{\beta}_{1}$ to $\tilde{y}$. 

This yields a pair of simple closed curves $(\alpha_{3},\beta_{3})$ on $S_{3,0}$ intersecting $2g$ times, and we claim that this is a filling pair. Indeed, let $\gamma$ be any simple closed curve on $S_{3,0}$ and assume $\gamma$ is disjoint from both $\alpha_{3}$ and $\beta_{3}$. Consider the projections $\pi^{\mathcal{AC}}_{\tilde{S}_{1}}(\gamma), \pi^{\mathcal{AC}}_{\tilde{S}_{2}}(\gamma)$ of $\gamma$ to the arc and curve complex $\mathcal{AC}$ of the subsurfaces $\tilde{S}_{1},\tilde{S}_{2}$. By assumption the arc $\pi^{\mathcal{AC}}_{\tilde{S}_{2}}(\gamma)$ is disjoint from the arcs $\tilde{x},\tilde{y}$.

 It then follows that this arc must be homotopic into $\partial \tilde{S}_{2}$, because the arcs $\tilde{x},\tilde{y}$ are distance at least $3$ in $\mathcal{AC}(\tilde{S}_{2})$. Hence $\gamma$ is homotopic into $\tilde{S}_{1}$; however, this contradicts the fact that $\tilde{\alpha},\tilde{\beta}$ fill $\tilde{S}_{1}$, and therefore $\gamma$ can not be disjoint from both $\alpha_{3}$ and $\beta_{3}$. 

Then to obtain a filling pair $(\alpha_{2k+1},\beta_{2k+1})$ intersecting $2(2k+1)$ times on any odd genus surface, we simply iterate this procedure by choosing a filling pair intersecting $2(2(k-1)+1)$ times on $S_{2(k-1)+1,0}$, cutting out a disk centered at any intersection point, and gluing on a copy of $\tilde{S}_{2}$. Thus, the existence of the desired pair for any even genus follows from the same argument by the existence of such a pair on $S_{2,0}$- see Figure $5$ below. This completes the proof of $(3)$.

 Then $(4)$ follows from $(2)$, and another application of the double bigon construction. 

Finally, the two upper bounds in $(5)$ are implied by the following construction on $S_{0,p}$, and the fact that any two simple closed curves must intersect an even number of times. 
\begin{figure}[H]
\centering
	\includegraphics[width=3.5in]{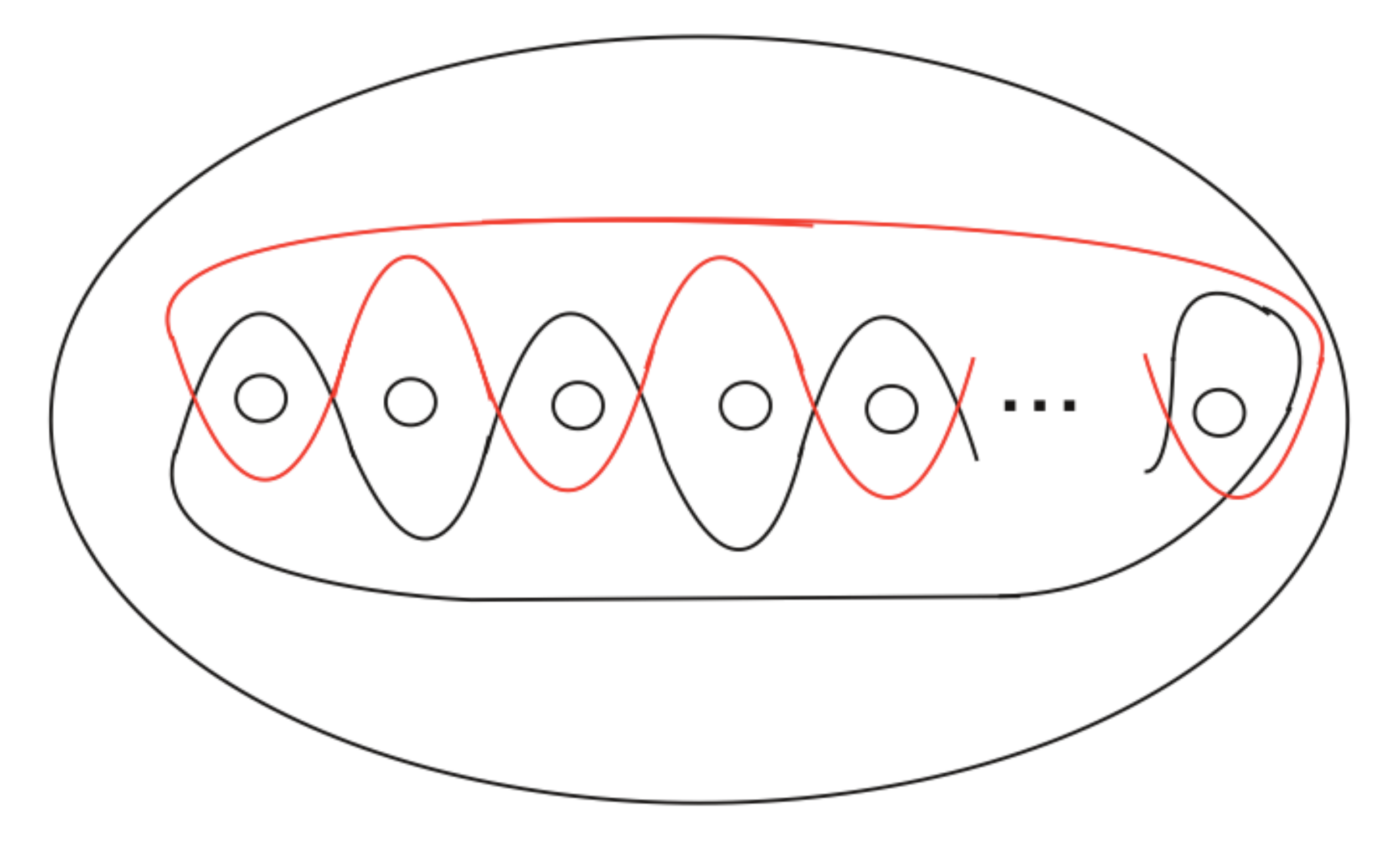}
\caption{The red and black simple closed curves fill, and, in the case $p$ is even, intersect $p$ times.}
\end{figure}

\end{proof}

To prove the main result, we will first exhibit the existence of a length $3$ geodesic segment, satisfying the property that any subsegment has endpoints intersecting close to minimally for their respective curve graph distances. The main theorem is then proved by carefully extending such a segment and inducting on curve graph distance. Thus, we conclude this section with the following lemma:

\begin{lemma} \label{Length $3$ paths} 
Given $S_{g,p}$, there exists a length $3$ geodesic segment $\left\{v_{0},v_{1},v_{2},v_{3}\right\}$ in $\mathcal{C}(S_{g,p})$ such that:
\begin{enumerate}
\item If $g\geq 1$, $g \neq 2$, then for any $k,j, 0\le k,j \le 3$, 
\[ i(v_{k},v_{j})=i_{|k-j|,(g,p)};\]
\item If $g=2$ and $p$ is even, $(1)$ holds. If $p$ is odd, then 
\[  i(v_{0},v_{3}) \le i_{3,(2,p)}+1, i(v_{0},v_{2})=i(v_{1},v_{3})=1.\]
\item If $g=0$ and $p$ is even,
\[ i(v_{0},v_{3})\le i_{3,(0,p)}+1, i(v_{0},v_{2})=2, i(v_{1},v_{3})=4;  \]
\item If $g=0$ and $p$ is odd, 
\[ i(v_{0},v_{3})=i_{3,(0,p)}, i(v_{0},v_{2})=2, i(v_{1},v_{3})=4.\]
\end{enumerate}
\end{lemma}
\begin{proof}
For $(1)$, assume first that $p=0$ and $g\neq 2$. Then by $(1)$ of Lemma $3.1$, there exists a filling pair $(\alpha,\beta)$ on $S_{g,p}$ whose complement consists of a single connected component. As in the proof of Lemma $3.1$, orient both $\alpha$ and $\beta$ and label the arcs along $\alpha$ (resp. $\beta$) $\alpha_{1},...,\alpha_{2g-1}$ (resp. $\beta_{1},...,\beta_{2g-1}$). Then cutting along $\alpha\cup \beta$ produces a single polygon $P$ with $(8g-4)$ sides, whose edges are labeled from the set 
\[A(g):= \left\{\alpha_{1}^{\pm},...,\alpha_{2g-1}^{\pm},\beta_{1}^{\pm},...,\beta_{2g-1}^{\pm}\right\}. \]
$(\alpha_{k}, \alpha_{k}^{-1})$ is referred to as an inverse pair; these edges project down to the same arc of $\alpha$ on the surface. Note that the edges of $P$ alternate between belonging to $\alpha$ and $\beta$.

Consider the map $M:A(g)\rightarrow A(g)$ which sends an edge $e$ to the inverse of the edge immediately following $e$ along $P$ in the clockwise direction. We claim that $M$ has order $4$. Indeed, the map $M$ is combinatorially an order $4$ rotation about an intersection point of $\alpha \cup \beta$, as pictured below. 

Now, suppose that every inverse pair constitutes a pair of opposite edges of $P$; that is to say, the complement of any inverse pair in the edge set of $P$ consists of two connected components with the same number of edges. Then $M$ induces a rotation of $P$ by $2\pi/(4g-1)$, which is not an order $4$ rotation, a contradiction.

\begin{figure}[H]
\centering
	\includegraphics[width=3in]{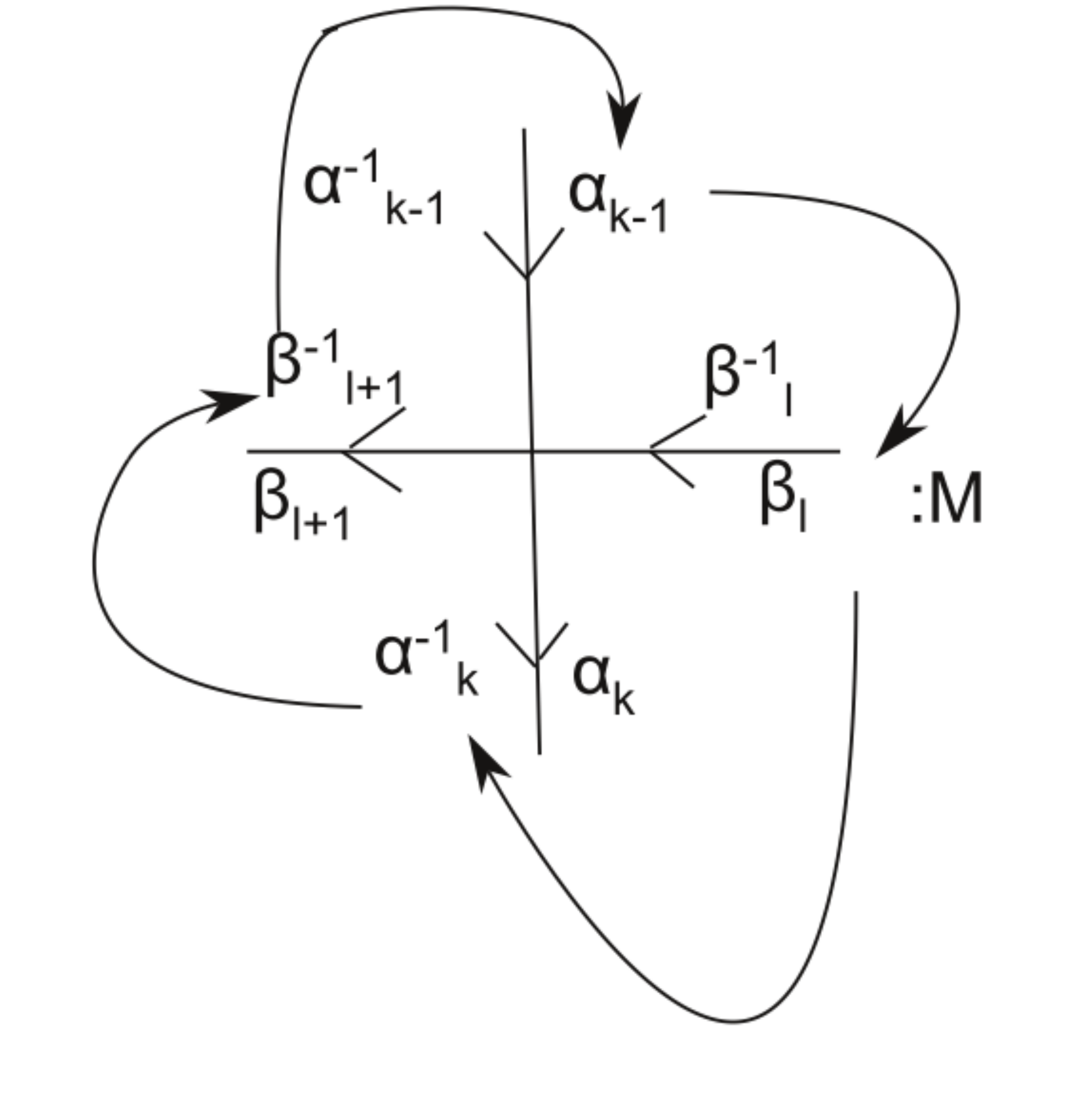}
\caption{$M$ sends the arc $\beta_{l}$ to $\alpha_{k}^{-1}$. The arrows demonstrate the order $4$ action of $M$ around the vertex.}
\end{figure} 

Therefore, there must be at least one inverse pair comprised of edges which are not opposite on $P$. Without loss of generality, this pair is of the form $(\alpha_{k},\alpha_{k}^{-1})$. Let $R$ be the connected component of the complement of $\alpha_{k} \cup \alpha_{k}^{-1}$ in the edge set of $P$ containing more than $4g-3$ edges.

Then there must exist an inverse pair of the form $(\beta_{j},\beta_{j}^{-1})$ contained in $R$, since the edges of $P$ alternate between belonging to $\alpha$ and $\beta$, and thus there must be a strictly larger number of $\beta$ edges in $R$ than in the other component. 

Then there is an arc connecting the edges $(\alpha_{k},\alpha_{k}^{-1})$ which projects down to a simple closed curve $v_{2}$ disjoint from $\beta$ and intersecting $\alpha$ exactly once. Similarly, there is an arc connecting $(\beta_{j},\beta_{j}^{-1})$ projecting down to a simple closed curve $v_{1}$ which is disjoint from both $v_{1}$ and $\alpha$, and which intersects $\beta$ exactly once. Then define $v_{0}:=\alpha, v_{3}:=\beta$; this concludes the proof of $(1)$ in the case $p=0$.

If $p>0$, then the double bigon construction introduced in the proof of Lemma $3.1$ can be used again here to obtain a geodesic segment $\left\{v_{0},...,v_{3}\right\}$ in $\mathcal{C}_{1}(S_{g,p})$ satisfying the desired property. 

If $g=2$, the existence of the desired geodesic segment in $\mathcal{C}_{1}(S_{2,0})$ will imply the existence of the corresponding segment in $\mathcal{C}_{1}(S_{2,p})$ for $p>1$ by another application of the double bigon construction. The filling pair $(\alpha,\beta)$ on $S_{2,0}$ shown on Page $41$ of \cite{FM} is obtained by gluing together a pair of octagons in accordance with the gluing pattern pictured below in Figure $4$. 

\begin{figure}[htbp]
\begin{center}
\includegraphics[width=3.5in]{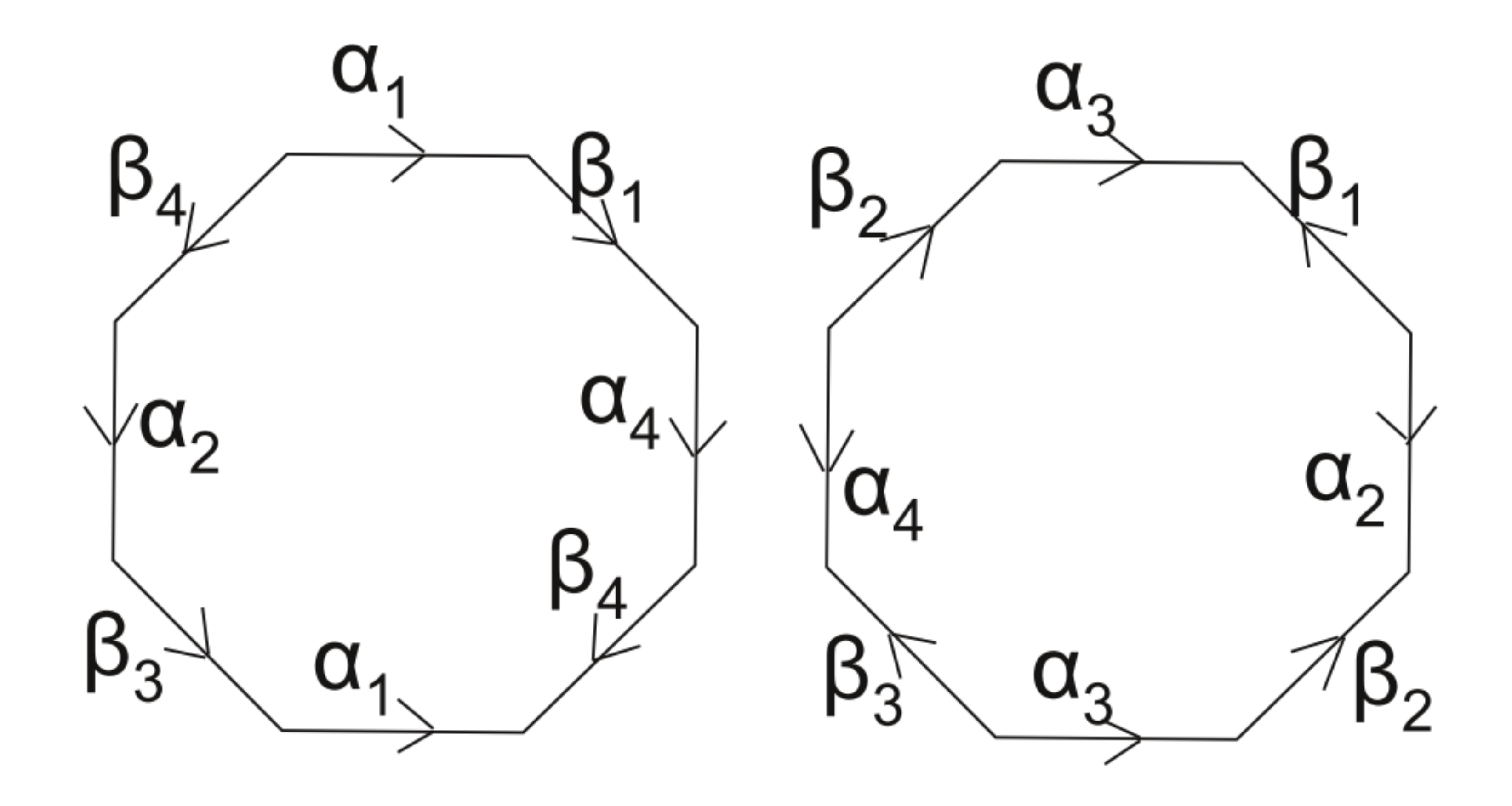}
\caption{Gluing the octagons together by pairing together sides with the same label produces $S_{2,0}$; the $\beta$-arcs concatenate in order to form a simple closed curve $\beta$, which fills $S_{2,0}$ with the simple closed curve $\alpha$- the concatenation of the $\alpha$-arcs.}
\label{default}
\end{center}
\end{figure}

Note that both $\alpha_{1}$ and $\alpha_{1}^{-1}$ are on the left octagon, and $\beta_{2},\beta_{2}^{-1}$ are both edges of the right octagon. Therefore, let $v_{3}$ be a simple closed curve whose lift to the disjoint union of octagons pictured above is an arc connecting $\alpha_{1}$ to $\alpha_{1}^{-1}$, and let $v_{2}$ be a curve whose lift is an arc connecting $\beta_{2}$ to $\beta_{2}^{-1}$. Then $\left\{\alpha=v_{0},v_{1},v_{2},v_{3}=\beta\right\}$ is the desired geodesic segment in $\mathcal{C}(S_{2,0})$. 

Finally, both $(3)$ and $(4)$ follow from the following picture on $S_{0,p}$:

\begin{figure}[h]
\centering
	\includegraphics[width=3.5in]{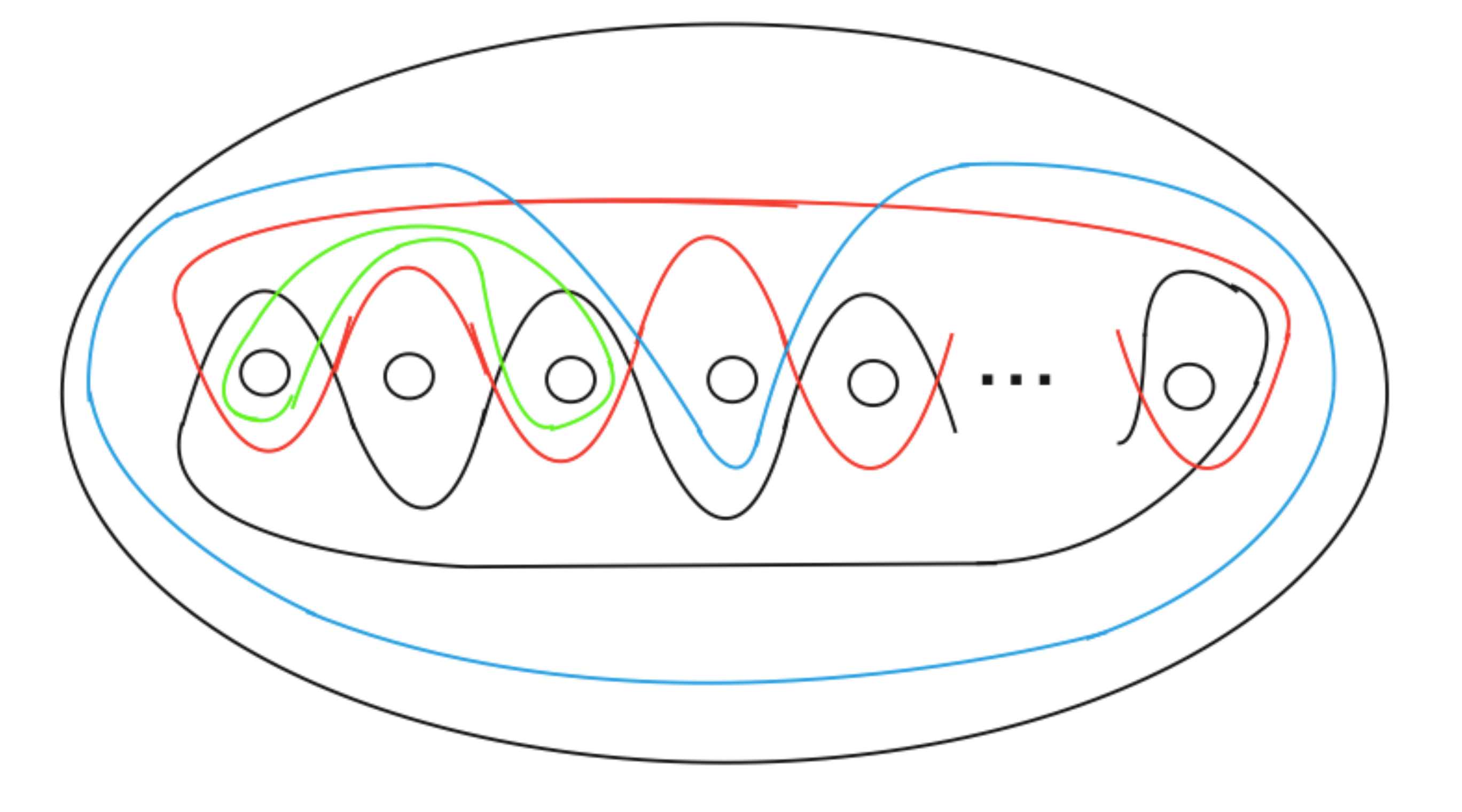}
\caption{Set the black curve equal to $v_{0}$, the blue to $v_{1}$, the green to $v_{2}$, and the red to $v_{3}$.}
\end{figure} 

\end{proof}

Since the intersection numbers determined in Lemma \ref{Length $3$ paths} are the basis for our construction in the next sections, we make the following notation: if $\{v_0,v_1,v_2,v_3\}$ is the geodesic in $\C(S_{g,p})$ determined by Lemma \ref{Length $3$ paths}, then set $\eta_{3,(g,p)} = i(v_0,v_3)$ and $\eta_{2,(g,p)} = \max \{i(v_1,v_3), i(v_0,v_2) \}$. Note that in most cases these are the minimum possible intersection numbers given their distance.

\section{Warm-up}

To give the idea of the general argument, we present a simplified proof of Theorem \ref{main} for the case where $n=4$. For small $n$, we can bypass the bounded geodesic image theorem using the simple fact that the projection from the curve complex to the curve complex of an annulus is coarsely $1$- Lipschitz.

Begin with curves $\alpha_3$ and $\beta_3$ that have distance $3$ in the curve complex and intersect $i_{3,(g,p)}$ times. Set $\beta_4$ equal to $T_{\alpha_3}^{10}(\beta_3)$ and $\alpha_4$ equal to $\beta_3$. Note that $d(\alpha_4, \beta_4) \le 6$ since each of these curves has distance $3$ from $\alpha_3$. If there is a geodesic from $\alpha_4$ to $\beta_4$ all of whose vertices intersect $\alpha_3$ then since the projection to $\C(\alpha_3)$ is Lipschitz
$$d_{\alpha_3} (\alpha_4, \beta_4)\le 6 + 1 = 7.  $$

This, however, contradicts our choice of Dehn twist $T_{\alpha_3}^{10}$ since
$$d_{\alpha_3}(\alpha_4,\beta_4) = d_{\alpha_3}(\beta_3, T_{\alpha_3}^{10}(\beta_3)) \ge 10 - 2 =8. $$

We conclude that any geodesic from $\alpha_4$ to $\beta_4$ must enter the one neighborhood of $\alpha_3$ and so $d(\alpha_4, \beta_4) \ge 4$. Finally, by the twist inequality
$$i(\alpha_4, \beta_4) = i(\beta_3, T_{\alpha_3}^{10}(\beta_3)) = 10 \cdot i(\alpha_3, \beta_3)^2 = 10 \cdot (i_{3,(g,p)})^2,$$
as required.

This process can be repeated, however at each step we require a twist whose power grows linearly with curve complex distance. To avoid this, we use the bounded geodesic image theorem.

\section{Minimal intersection rays}
Set $B = M+3$, where $M$ is as in Theorem \ref{BGIT}. Fix a surface $S = S_{g,p}$ and begin with the length $3$ geodesic $[v_0,v_1,v_2,v_3]$ in $\C(S_{g,p})$ with $i(v_i,v_j) = \eta_{|j-i|,(g,p)}$ as in Lemma \ref{Length $3$ paths} and the final paragraph of Section \ref{MIFC}. Set $\eta = \eta_{3,(g,p)}$. What's important here is the fact that $i(v_0,v_3)$ is bounded linearly in the complexity of $S$, while $i(v_0,v_2)$ and  $i(v_1,v_3)$ are uniformly bounded, independent of complexity. From this, we construct a geodesic ray whose vertices have optimal intersection number given their distance, in the sense described in the introduction. We began by defining a sequence of geodesics $\gamma_k$ in $\C(S)$ whose lengths grow exponentially in $k$ and have the property that all but the last vertex of $\gamma_k$ is contained in $\gamma_{k+1}$. We refer to $k$ as the \emph{level} of $\gamma_k$.

Set $\gamma_0 = [v_0,v_1,v_2,v_3]$. To construct $\gamma_{k+1}$ from $\gamma_k$, let $e_k$ be the terminal vertex of $\gamma_k$ that is not $v_0$ and set $\gamma'_k$ to be $\gamma_k$ minus the vertex $e_k$. Then define
$$\gamma_{k+1} = \gamma'_k \cup  T^B_{e_k}(\gamma'_k),$$
where $B = M+3$. We think of $\gamma_{k+1}$ as a edge path from $v_0$ to $T_{e_k}^B(v_0)$ so that our recursive definition makes sense. An simple argument shows that the length of $\gamma_k$ is $\ell(\gamma_k) = 2^k+2$ and that $\gamma'_k$ is an initial subgeodesic of $\gamma_{k+1}$.

\begin{lemma}
For $k \ge 0$, $\gamma_k$ is a geodesic in $\C(S)$.
\end{lemma}

\begin{proof}
For $k=0$ this is by construction. Assume that the lemma holds for $\gamma_k$ and recall that $\gamma_{k+1} = \gamma'_k \cup  T^B_{e_k}(\gamma'_k)$ has length $2^{k+1}+2$. Note that 
$$d_{e_k}(v_0, T^B_{e_k}(v_0) \ge B -2 > M, $$
so by the bounded geodesic image theorem any geodesic between these vertices must pass through a $1$-neighborhood of $e_k$. Hence, $d(v_0, T^B_{e_k}(v_0)) \ge 2(2^k +2) -2 = 2^{k+1}+2 = \ell(\gamma_{k+1})$. Hence, $\gamma_{k+1}$ is a geodesic.
\end{proof}

The following theorem is our main technical result. It gives the desired intersection number, by \emph{level}. The corollary following it removes the dependence on level.

\begin{theorem}
For $k\ge0$ and $\gamma_k = [v_0, \ldots, v_{2^k+2}]$ the following inequality holds for all $0 \le i,j\le 2^k +2$:
$$i(v_i,v_j) \le \epsilon (\epsilon B)^{2^k -1}\eta^{|j-i|-2} + O(\eta^{|j-i|-4}), $$
where $\epsilon =1$ if $g \ge 1$ and $\epsilon = 4$ otherwise.
\end{theorem}

\begin {proof}
The proof is by induction on $k$. For $k=0$, this holds by our choice of $[v_0,v_1,v_2,v_3]$ using Lemma \ref{Length $3$ paths}. Suppose the result holds for $\gamma_k$ and to simplify notation set $n = 2^k+2$ relabel the vertices of $\gamma_k$ so that $\gamma_k = [v_0, \ldots, v_{n-1}, v_{n}].$ With this notation
$$\gamma_{k+1} = [v_0, \ldots, v_{n-1}] \cup [T^B_{v_n}(v_{n-1}), \ldots, T^B_{v_n}(v_0)]$$
where $v_{n-1} = T^B_{v_n}(v_{n-1})$.

By the induction hypotheses, it suffices to bound intersections of the form
$$i(T^B_{v_n}(v_j),w),$$
where $v_j \in \gamma_k$ and $w \in \gamma_{k+1}$. If $w = T^B_{v_n}(w')$ for some $w' \in \gamma_k$, then we observe that $i(T^B_{v_n}(v_j),w) = i(v_j,w')$ and apply the induction hypothesis at level $k$. Therefore, we may assume that $w= v_i \in \gamma_k$ and $i \le n-2$.

Using the twist inequality, we apply the induction hypotheses to $\gamma_k$ and compute 
\begin{eqnarray*}
i(T^B_{v_n}(v_j),v_i) &\le& B i(v_n,v_j) i(v_n,v_i) + i(v_j,v_i) \\
&\le& B (\epsilon (\epsilon B)^{2^k-1}\eta^{|n-j|-2} +O(\eta^{|n-j|-4}) ) \cdot \\
 && (\epsilon (\epsilon B)^{2^k-1}\eta^{|n-i|-2} +O(\eta^{|n-i|-4})) +\\
 && (\epsilon (\epsilon B)^{2^k-1}\eta^{|j-i|-2} +O(\eta^{|j-i|-4}))\\
 &\le& \epsilon (\epsilon B)^{2^{k+1}+2} \eta^{(2n-j-i-4)}+ O(\eta^{(2n-i-j-6)}).
\end{eqnarray*}
Here we have used our assumption that $i,j \le n-2$ to conclude that $\eta^{{(2n-i-j-6)}}$ dominates $\eta^{(i-j-2)}$. Since $\gamma_{k+1}$ is a geodesic, the distance from $T^B_{v_n}(v_j)$ to $v_i$ is $2n-j-i-2$. If we denote this distance by $d$, we have shown
$$i(T^B_{v_n}(v_j),v_i) \le \epsilon (\epsilon B)^{2^{k+1}+2}\eta^{d-2} + O(\eta^{d-4}).$$
This completes the proof.

\end{proof}

Now set $\gamma = \cup_k \gamma'_k$. This is an infinite geodesic ray with endpoint $v_0$. For convenience, relabel the vertices of $\gamma$ so that $\gamma = [v_0,v_1,v_2, \ldots]$.

\begin{corollary}
Let $\gamma$ be the geodesic ray in $\C(S_{g,p})$ as described above. Then for any $i \le j$
$$i(v_i,v_j) \le \epsilon (\epsilon B)^{2j-5} \eta^{|j-i|-2} + O(\eta^{|j-i|-4}),$$
where $\epsilon =1$ if $g \ge 1$ and $\epsilon = 4$ otherwise.
\end{corollary}

\begin{proof}
Take $k$ so that $2^k+2 \le j < 2^{k+1} +2.$ Then $v_j$ is a vertex of $\gamma$ that first appears at the $k+1$th level. That is, $v_j \in \gamma_{k+1}$. Then $2^{k+1}-1 = (2^{k+1} + 4) -5 \le 2j-5$. Now apply the main theorem to $\gamma_{k+1}$ conclude the proof.
\end{proof}

\bibliographystyle{amsalpha}
%

\bibliography{SmallIntersection.bbl}

\end{document}